\documentclass{amsart}
\usepackage{amsfonts}
\usepackage{graphicx}

\newtheorem{theorem}{Theorem}[section]
\newtheorem*{theorem A}{Theorem A}
\newtheorem*{theorem B}{N\"olker's Theorem}

\theoremstyle{remark}

\theoremstyle{remark}

\theoremstyle{definition}

\numberwithin{equation}{section}
\def\({\left ( }
\def\){\right )}
\def\<{\left < }
\def\>{\right >}

\input{tcilatex}
\usepackage{rotating}

\begin{document}
\title[Constant curvature translation surfaces in Galilean 3-space] {Constant curvature translation surfaces in Galilean 3-space}
\author{Alper Osman Ogrenmis$^1$, Mihriban Kulahci $^2$, Muhittin Evren Aydin$^3$}
\address{$^{1,2,3}$ Department of Mathematics, Faculty of Science, Firat University, Elazig, 23200, Turkey}
\email{aogrenmis@firat.edu.tr, mihribankulahci@gmail.com, meaydin@firat.edu.tr}
\thanks{}
\subjclass[2000]{53A35, 53B25, 53C42.}
\keywords{Galilean space, translation surface, Gaussian curvature,
mean curvature.}

\begin{abstract}
Total five different types of translation surfaces, based upon \linebreak planarity of translating curves and the absolute figure, arise
in a Galilean 3-space. Excepting
the type in which both of translating curves are non-planar we obtain these
surfaces with arbitrary constant Gaussian and mean curvature.
\end{abstract}

\maketitle

\section{Introduction and Preliminaries}

The \textit{translation surfaces,} among the family of surfaces in classic
differential \linebreak geometry, have been commonly examined since early 1900s and for
that reason an extensive literature relating to these appears. For example
see \cite{3,5,6}, \cite{12}-\cite{18}, \cite{24}-\cite{26}, \cite{31}-\cite%
{36}. Such surfaces are geometrically described as translating two curves along each other up
to isometries of the ambient space. As far as we know the counterparts of
this notion in a Galilean space $\mathbb{G}_{3}$ were firstly considered in
\linebreak  Sipus and Divjak's work \cite{20} by providing translation surfaces with
constant \linebreak  Gaussian $\left( K\right) $ and mean curvature $\left( H\right) $
under the restriction that the translating curves lie in orthogonal planes.
Extending this restriction, which is our motivation for the present study, leads us to open fields for further
investigations. More \linebreak precisely, by assuming $K=const.$ and $H=const.$ we
shall present the translation surfaces in $\mathbb{G}_{3},$ except the ones
whose both of translating curves are space curves.

A \textit{Cayley-Klein 3-space} is defined as a projective 3-space $P_{3}\left( 
\mathbb{R}\right) $ with certain absolute figure. \textit{Group of motions} of this space are introduced by the projective
transformations which leave invariant the absolute figure. Metrically
arguments given up to the absolute figure are invariant under this
group (cf. \cite{23}). The Galilean 3-space $\mathbb{G}_{3}$ is one of real Cayley-Klein 3-spaces with the
{\it absolute figure} $\left\{ \Gamma ,l,\iota \right\} ,$ where $\Gamma $ is a plane ({\it absolute plane}) in $P_{3}\left( \mathbb{R}\right) $, $l$ a line ({\it absolute line}) in $\Gamma $ and $\iota $ is the {\it fixed elliptic involution} of the
points of $l$. For technical details, we refer the reader to \cite{1,2,4},  
\cite{7}-\cite{10}, \cite{19,21,22} \cite{27}-\cite{30}, \cite{37}. Let $\left(
x_{0}:x_{1}:x_{2}:x_{3}\right) $ denote the homogeneous coordinates in $%
P_{3}\left( \mathbb{R}\right) .$ Then $\Gamma $ is characterized by $x_{0}=0,$ $l$ by $x_{0}=x_{1}=0$ and $\iota $ by 
\begin{equation*}
\left( x_{0}:x_{1}:x_{2}:x_{3}\right)
\longmapsto \left( x_{0}:x_{1}:x_{3}:-x_{2}\right) .
\end{equation*}
Passing from the homogeneous coordinates to the affine coordinates is essential to 
introduce the affine model of $\mathbb{G}_{3}$ that is our interest field.
Then, by means of the affine coordinates, the \textit{group of motions} of $\mathbb{G}_{3}$ is
given by the transformation

\begin{equation*}
\left( x,y,z\right) \longmapsto \left( x^{\prime },y^{\prime },z^{\prime
}\right) :\left\{ 
\begin{array}{l}
x^{\prime }=a+x, \\ 
y^{\prime }=b+cx+\left( \cos \theta \right) y+\left( \sin \theta \right) z,
\\ 
z^{\prime }=d+ex-\left( \sin \theta \right) y+\left( \cos \theta \right) z,%
\end{array}%
\right.
\end{equation*}%
where $a,b,c,d,e$ and $\theta $ are some constants. For given points $%
X=\left( x_{1},x_{2},x_{3}\right) $ and $Y=\left( y_{1},y_{2},y_{3}\right) ,$
the \textit{Galilean distance} is introduced by the absolute figure, namely%
\begin{equation*}
d\left( X,Y\right) =\left\{ 
\begin{array}{ll}
\left\vert y_{1}-x_{1}\right\vert , & \text{if }x_{1}\neq 0\text{ or }%
y_{1}\neq 0, \\ 
\sqrt{\left( y_{2}-x_{2}\right) ^{2}+\left( y_{3}-x_{3}\right) ^{2}}, & 
\text{if }x_{1}=0\text{ and }y_{1}=0.%
\end{array}%
\right.
\end{equation*}

The lines and planes are categorized up to the absolute figure. Explicitly,
a line is said to be \textit{non-isotropic} (resp. \textit{isotropic}) if
its intersection with the absolute line $l$ is empty (resp. non-empty).
Contrary to this, a plane is said to be \textit{isotropic} if it
does not involve $l$, otherwise it is said to be\textit{\ Euclidean}. In
other words, an isotropic plane does not involve any isotropic direction. In
the affine model of $\mathbb{G}_{3}$, the Euclidean planes are determined by the equation $%
x=const.$ Accordingly, a vector is called \textit{isotropic} if it is
involved in the Euclidean plane $x=0$. \textit{Non-isotropic vectors} are of
the form $\left( a\neq 0,b,c\right) .$

A curve given in parametric form $\alpha =\alpha (s)=(x\left( s\right)
,y(s),z(s))$ is said to be \textit{non-isotropic} (or \textit{admissible})
if nowhere its tangent vector is isotropic, namely $x^{\prime }\left(
s\right) =\frac{dx}{ds}\neq 0$. Otherwise the curve $\alpha $ is said to be 
\textit{isotropic}. If $\alpha $ is a non-isotropic curve having unit speed
(i.e. $x^{\prime }\left( s\right) =\pm 1$), then the \textit{curvature} and 
\textit{torsion} are given by%
\begin{equation*}
\kappa (s)=\sqrt{\left[ y^{\prime \prime }(s)\right] ^{2}+\left[ z^{\prime
\prime }(s)\right] ^{2}},\text{ \ \ }\tau (s)=\frac{\det \left (\alpha ^{\prime
}(s),\alpha ^{\prime \prime }(s),\alpha ^{\prime \prime \prime }(s) \right )}{\left[
\kappa \left( s\right) \right] ^{2}}\text{ }\left( \kappa \left( s\right)
\neq 0\right) .
\end{equation*}%
We call a curve \textit{planar} (resp. \textit{space curve}) provided $\tau
(s)=0$ $\left( \text{resp. }\tau (s)\neq 0\right) $ for all $s.$ Obviously,
the space curves are non-isotropic, whereas the isotropic curves are Euclidean
planar, that is, lie in a Euclidean plane.

A regular surface immersed in $\mathbb{G}_{3}$ is parameterized by the
mapping%
\begin{equation*}
r:D\subseteq \mathbb{R}^{2}\longrightarrow \mathbb{G}_{3},\text{ }\left(
u_{1},u_{2}\right) \longmapsto \left( x\left( u_{1},u_{2}\right) ,y\left(
u_{1},u_{2}\right) ,z\left( u_{1},u_{2}\right) \right) .
\end{equation*}%
In order to specify the partial derivatives we shall notate: 
\begin{equation*}
x_{,i}=\frac{\partial x}{\partial u_{i}}\text{ and }x_{,ij}=\frac{\partial
^{2}x}{\partial u_{i}\partial u_{j}},\text{ }1\leq i,j\leq 2.
\end{equation*}%
Then $r$ is said to satisfy \textit{admissibility} criteria if nowhere it
has Euclidean tangent planes, i.e., $x_{,i}\neq 0$ for some $i=1,2.$ The
first fundamental form is given by%
\begin{equation*}
ds^{2}=\left( g_{1}du_{1}+g_{2}du_{2}\right) ^{2}+\varepsilon \left(
h_{11}du_{1}^{2}+2h_{12}du_{1}du_{2}+h_{22}du_{2}^{2}\right) ,
\end{equation*}%
where $g_{i}=x_{,i}$, $h_{ij}=y_{,i}y_{,j}+z_{,i}z_{,j},$ $i,j=1,2,$ and 
\begin{equation*}
\varepsilon =\left\{ 
\begin{array}{ll}
0, & \text{if the direction }du_{1}:du_{2}\text{ is non-isotropic,} \\ 
1, & \text{if the direction }du_{1}:du_{2}\text{ is isotropic.}%
\end{array}%
\right.
\end{equation*}

Let us introduce a function $W$ given by%
\begin{equation*}
W=\sqrt{\left( x_{,1}z_{,2}-x_{,2}z_{,1}\right) ^{2}+\left(
x_{,2}y_{,1}-x_{,1}y_{,2}\right) ^{2}}.
\end{equation*}%
Then the normal vector field is defined as%
\begin{equation*}
N=\frac{1}{W}\left(
0,-x_{,1}z_{,2}+x_{,2}z_{,1},x_{,1}y_{,2}-x_{,2}y_{,1}\right)
\end{equation*}%
and thereafter the second fundamental form%
\begin{equation*}
II=L_{11}du_{1}^{2}+2L_{12}du_{1}du_{2}+L_{22}du_{2}^{2},
\end{equation*}%
where%
\begin{equation*}
L_{ij}=\frac{1}{g_{1}}\left( g_{1}\left( 0,y_{,ij},z_{,ij}\right)
-g_{i,j}\left( 0,y_{,1},z_{,1}\right) \right) \cdot N,\text{ }g_{1}\neq 0
\end{equation*}%
or%
\begin{equation*}
L_{ij}=\frac{1}{g_{2}}\left( g_{2}\left( 0,y_{,ij},z_{,ij}\right)
-g_{i,j}\left( 0,y_{,2},z_{,2}\right) \right) \cdot N,\text{ }g_{2}\neq 0.
\end{equation*}%
Note that the dot $"\cdot "$ denotes the \textit{Euclidean scalar product}.
Thereby, the \textit{Gaussian }and \textit{mean curvature} are defined as%
\begin{equation*}
K=\frac{L_{11}L_{22}-L_{12}^{2}}{W^{2}}\text{ and }H=\frac{%
g_{2}^{2}L_{11}-2g_{1}g_{2}L_{12}+g_{1}^{2}L_{22}}{2W^{2}}.
\end{equation*}%
A surface is said to be \textit{minimal} (resp. \textit{flat}) if its mean
(resp. Gaussian) curvature vanishes. Recall that the minimal surfaces in $\mathbb{G}_{3}$ were
classified in \cite{29} by the result:

\begin{theorem} 
Minimal surfaces in $\mathbb{G}_{3}$ are cones whose vertices lie on the absolute line and the ruled surfaces of type C. They are all conoidal ruled surfaces having the absolute line as the directional line in infinity.
\end{theorem}

Recall that a ruled surface of type C is of the form $r(u,v)= \left (u,x(u)+vy(u),vz(u) \right )$.
\section{Translation Surfaces}

A \textit{translation surface} in $\mathbb{G}_{3}$  is
locally parameterized by 
\begin{equation*}
r:I_{1}\times I_{2}\subseteq \mathbb{R}^{2}\longrightarrow \mathbb{G}_{3},%
\text{ }r\left( x,y\right) =\alpha \left( x\right) +\beta \left( y\right) ,
\end{equation*}
where $\alpha $ and $\beta $ denote \textit{translating curves. }Under the
condition that $\alpha $ and $\beta $ are planar, the authors in \cite{20}
categorized such a surface up to the absolute figure:

\begin{description}
\item[type 1] $\alpha $ is planar non-isotropic curve and $\beta $ isotropic
curve,

\item[type 2] $\alpha $ and $\beta $ are planar non-isotropic curves.
\end{description}
If the planes involving translating curves are chosen to be mutually
orthogonal, the surfaces of type 1 and type 2 have the parametrizations,
respectively%
\begin{equation}
r\left( x,y\right) =\left( x,y,f\left( x\right) +g\left( y\right) \right) 
\text{ \ and\  \ }r\left( x,y\right) =\left( x+y,g\left( y\right) ,f\left(
x\right) \right) .  \tag{2.1}
\end{equation}%
These surfaces with $K=const.$ and $H=const.$ were obtained in \cite{20}. If
not, i.e. the planes are non-orthogonal, then the notion of \textit{affine
translation surface} naturally arises, that firstly introduced by Liu and Yu 
\cite{14} as the graph surfaces of the functions%
\begin{equation*}
z\left( x,y\right) =f\left( x\right) +g\left( y+ax\right) ,\text{ }a\neq 0.
\end{equation*}%
By following this, the surfaces of type 1 and type 2 are generally called 
\textit{affine translation surfaces.} We shall classify such surfaces in
Section 3 with $K=const.$ and $H=const.$ Furthermore, the translating curves
could be non-planar and hereinafter it is necessary to extend above
categorization:

\begin{description}
\item[type 3] $\alpha $ is isotropic curve and $\beta $ space curve,

\item[type 4] $\alpha $ is planar non-isotropic curve and $\beta $ space
curve,

\item[type 5] $\alpha $ and $\beta $ are space curves.
\end{description}
We shall also provide the surfaces of type 3 and type 4 in next sections
with $K=const.$ and $H=const.$ 

\section{Constant Curvature Affine Translation Surfaces}

Assume that $A=(a_{ij})$ is a regular real matrix, $i,j=1,2,$ and $w=\det
A\neq 0$. Let us consider the following planar curves:%
\begin{equation}
\left\{ 
\begin{array}{cc}
\alpha=\alpha (u)=\left( \dfrac{a_{22}}{w}u,\dfrac{-a_{21}}{w}u,f(u)\right) , & 
P_{\alpha }:a_{21}x+a_{22}y=0, \\ 
\beta =\beta (v)=\left( \dfrac{-a_{12}}{w}v,\dfrac{a_{11}}{w}v,g(v)\right) , & 
P_{\beta }:a_{11}x+a_{12}y=0,%
\end{array}%
\text{ }\right.  \tag{3.1}
\end{equation}%
where $P_{\alpha }$ and $P_{\beta }$ denotes the planes involving the
curves. It is easily seen that $P_{\alpha }$ is orthogonal to $P_{\beta }$
provided $A$ is an orthogonal matrix. If $a_{12}=0$ (resp. $a_{22}=0$) in
(3.1) then $\beta $ (resp. $\alpha $) becomes an isotropic curve. Otherwise
both of them are non-isotropic curves. Therefore, by a translation of $%
\alpha $ and $\beta ,$ we derive the following admissible surface 
\begin{equation}
r(u,v)=\left( \frac{a_{22}}{w}u-\frac{a_{12}}{w}v,\frac{a_{11}}{w}v-\frac{%
a_{21}}{w}u,f(u)+g(v)\right) .  \tag{3.2}
\end{equation}%
By changing the coordinates $u=a_{11}x+a_{12}y,$ \ $v=a_{21}x+a_{22}y,$
(3.2) turns to the standart parametrization of \textit{affine translation surface}
given by%
\begin{equation}
r(x,y)=\left( x,y,f(a_{11}x+a_{12}y)+g(a_{21}x+a_{22}y)\right) .  \tag{3.3}
\end{equation}%
This one represents the surfaces of both type 1 and type 2 as well as a
natural generalization of the surfaces given by (2.1). Throughout this
section, we shall only distinguish the cases relating to $f$ due to the fact
that the roles of $f$ and $g$ are symmetric. After a calculation, we have
the Gaussian curvature:%
\begin{equation}
K=\frac{w^{2}f^{\prime \prime }g^{\prime \prime }}{\left[ 1+\left(
a_{12}f^{\prime }+a_{22}g^{\prime }\right) ^{2}\right] },  \tag{3.4}
\end{equation}%
where $f^{\prime }=\frac{df}{du}$ and $g^{\prime }=\frac{dg}{dv},$ etc.

\begin{theorem} If an affine translation surface given by (3.3) has
constant Gaussian curvature $K_{0}$ in $\mathbb{G}_{3}$, then it is either

\begin{enumerate}
\item[(1)] a generalized cylinder with isotropic or non-isotropic rulings ($K_{0}=0$);
\item[(2)]  or a certain surface parameterized by, up to suitable
translations and constants, 
\begin{equation*}
r(x,s)=\left( x,c_{1}x+\frac{K_{0}}{c_{2}}s^{2},c_{3}x^{2}+%
\frac{1}{2}s\sqrt{1-\frac{K_{0}}{c_{2}}s^{2}}+\sqrt{%
\frac{c_{2}}{16K_{0}}}\arcsin \left( \sqrt{\frac{4K_{0}}{c_{2}}}s\right)
\right) ,
\end{equation*}%
where $c_{1},c_{2},c_{3}\in \mathbb{R}-\left\{ 0\right\} $\ and $s$ is the arc-length parameter of $\beta$.
\end{enumerate}
\end{theorem}

\begin{proof} Assume that $K_{0}=0.$ Then (3.4) leads $f$ to be a
linear function and thus the surface becomes a generalized cylinder
(so-called cylindrical surface, see \cite{11}, p. 439). Otherwise, i.e. $%
K_{0}\neq 0,$ by (3.4) we get $f^{\prime \prime }g^{\prime \prime }\neq 0.$
Taking the partial derivative of (3.4) with respect to $u$ gives%
\begin{equation}
4K_{0}[1+\left( a_{12}f^{\prime }+a_{22}g^{\prime }\right) ^{2}]\left[
a_{12}f^{\prime }+a_{22}g^{\prime }\right] [a_{12}f^{\prime \prime
}]=w^{2}f^{\prime \prime \prime }g^{\prime \prime }.  \tag{3.5}
\end{equation}%
To solve (3.5), we have two cases:

\begin{itemize}
\item[\textbf{Case (A)}] $a_{12}=0.$ (3.5) follows that $f^{\prime \prime
}=c_{1}\neq 0.$ Then by (3.4) we get%
\begin{equation}
\frac{K_{0}}{a_{11}^{2}a_{22}}=\frac{a_{22}g^{\prime \prime }}{\left[
1+(a_{22}g^{\prime })^{2}\right] ^{2}},  \tag{3.6}
\end{equation}%
where $a_{11}a_{22}\neq 0$ since $w\neq 0.$ We treat the method used in \cite%
{11} in order to solve (3.6). Since $a_{12}=0,$ $\beta $ is an isotropic curve and its reparametrization
having unit speed is given by 
\begin{equation}
\beta (s)=(0,p(s),q(s)),\text{ }\left( p^{\prime }\right) ^{2}+\left(
q^{\prime }\right) ^{2}=1,  \tag{3.7}
\end{equation}%
where the prime denotes the derivative with respect to the arc-length parameter. In this case (3.4) turns to 
\begin{equation}
K_{0}=f^{\prime \prime }q^{\prime \prime }.  \tag{3.8}
\end{equation}%
After solving (3.8), up to suitable translations and constants, we deduce $q=%
\frac{K_{0}}{c_{1}}s^{2}.$ Considering it into (3.7) leads to%
\begin{equation*}
p(s)=\frac{1}{2}s\sqrt{1-\frac{K_{0}}{c_{1}}s^{2}}+\frac{1}{4}\sqrt{\frac{%
c_{1}}{K_{0}}}\arcsin \left( 2\sqrt{\frac{K_{0}}{c_{1}}}s\right) ,
\end{equation*}%
which proves the second statement of the theorem.

\item[\textbf{Case (B)}] $a_{12}\neq 0.$ By the symmetry we have $a_{22}\neq 0$ and
then (3.5) can be rewritten as%
\begin{equation}
\left[ 1+\left( a_{12}f^{\prime }+a_{22}g^{\prime }\right) ^{2}\right] \left(
a_{12}f^{\prime }+a_{22}g^{\prime }\right) (g^{\prime \prime })^{-1}=\frac{%
w^{2}f^{\prime \prime \prime }}{4K_{0}a_{12}f^{\prime \prime }}.  \tag{3.9}
\end{equation}%
The partial derivative of (3.9) with respect to $v$ yields%
\begin{equation}
\frac{1+3\left( a_{12}f^{\prime }+a_{22}g^{\prime }\right) ^{2}}{%
a_{12}f^{\prime }+a_{22}g^{\prime }+\left( a_{12}f^{\prime }+a_{22}g^{\prime
}\right) ^{3}}=\frac{g^{\prime \prime \prime }}{a_{22}(g^{\prime \prime
})^{2}}.  \tag{3.10}
\end{equation}%
Again taking the partial derivative of (3.10) with respect to $u$ gives the
following polynomial equation on $\left( a_{12}f^{\prime }+a_{22}g^{\prime
}\right):$ 
\begin{equation*}
1+3\left( a_{12}f^{\prime }+a_{22}g^{\prime }\right) ^{4}=0,
\end{equation*}%
which is a contradiction and completes the proof.
\end{itemize}
\end{proof} 

For the mean curvature, we have%
\begin{equation}
H=\frac{a_{12}^{2}f^{\prime \prime }+a_{22}^{2}g^{\prime \prime }}{\left[
1+\left( a_{12}f^{\prime }+a_{22}g^{\prime }\right) ^{2}\right] ^{\frac{3}{2}%
}}.  \tag{3.11}
\end{equation}

\begin{theorem} Let an affine translation surface given by (3.3) have
constant mean curvature $H_{0}$\ in $\mathbb{G}_{3}.$\ Then:

\begin{enumerate}
\item[(1)] If $H_{0}=0,$\ it is either

\begin{enumerate}
\item[(1.1)] an isotropic plane, or

\item[(1.2)] a generalized cylinder with isotropic rulings, or

\item[(1.3)] a non-cylindrical ruled surface of type C whose the base curve is a parabolic circle.
\end{enumerate}

\item[(2)] Otherwise ($H_{0}\neq 0$); it is either

\begin{enumerate}
\item[(2.1)] a certain surface given by%
\begin{equation*}
r(x,y)=\left( x,y,f(a_{11}x)-\frac{1}{H_{0}}\sqrt{1-\left( \frac{%
H_{0}}{a_{22}}v\right) ^{2}}\right)  , \text{ }a_{22}\neq 0,
\end{equation*}

\item[(2.2)] or a generalized cylinder with non-isotropic rulings given by%
\begin{equation*}
r(x,y)=\left( x,y,\frac{c_{1}w}{a_{22}}x-\frac{1}{H_{0}}\sqrt{1-\left( \frac{%
H_{0}}{a_{22}}v\right) ^{2}}\right) ,\text{ }a_{22}\neq 0,
\end{equation*}
\end{enumerate}
\end{enumerate}

where $v=a_{21}x+a_{22}y.$
\end{theorem}

\begin{proof} We divide the proof into two cases:

\begin{itemize}
\item[\textbf{Case (A)}] $H_{0}=0.$ Then (3.11) reduces to 
\begin{equation}
a_{12}^{2}f^{\prime \prime }+a_{22}^{2}g^{\prime \prime }=0.  \tag{3.12}
\end{equation}%
We have again cases:

\begin{itemize}
\item[\textbf{Case (A.i)}] $f^{\prime \prime }=0=g^{\prime \prime }$ is a
solution for (3.12). This leads the surface to be an isotropic plane, which
implies the statement (1.1) of the thorem.

\item[\textbf{Case (A.ii)}] $a_{12}=0.$ Since $w\neq 0,$ we get $a_{22}\neq
0 $. Thus (3.12) immediately implies $g^{\prime \prime }=0,$ which proves
the statement (2) of the theorem.

\item[\textbf{Case (A.iii)}] $a_{12}\neq 0.$ The symmetry implies $%
a_{22}\neq 0.$ Solving (3.12) gives, up to suitable translations and
constants, 
\begin{equation*}
f(u)=\frac{c_{1}}{2a_{12}^{2}}u^{2},\quad g(v)=-\frac{c_{1}}{2a_{22}^{2}}%
v^{2}.
\end{equation*}%
Substituting this into (3.3) gives%
\begin{equation*}
r(x,y)=\left( x,0,\frac{c_{1}}{2}\left[ \left( \frac{a_{11}}{a_{12}}\right)
^{2}-\left( \frac{a_{21}}{a_{22}}\right) ^{2}\right] x^{2}\right) +y\left(
0,1,2x\left[ \frac{a_{11}}{a_{12}}-\frac{a_{21}}{a_{22}}\right] \right) ,
\end{equation*}%
which parametrizes the non-cylindrical  ruled surface whose the base curve is a parabolic circle and
the rulings are isotropic.
\end{itemize}

\item[\textbf{Case (B)}] $H_{0}\neq 0.$ We have two cases:

\begin{itemize}
\item[\textbf{Case (B.i)}] $a_{12}=0.$ Then (3.11) reduces to%
\begin{equation}
H_{0}=\frac{a_{22}^{2}g^{\prime \prime }}{\left[ 1+\left( a_{22}g^{\prime
}\right) ^{2}\right] ^{\frac{3}{2}}}.  \tag{3.13}
\end{equation}%
After solving (3.13), up to suitable translations and constants, we deduce%
\begin{equation*}
g\left( a_{22}y\right) =-\frac{1}{H_{0}}\sqrt{1-\left(H_{0}y\right) ^{2}},
\end{equation*}%
where $a_{22}\neq 0$ since $w\neq 0.$ This proves the statement (2.1) of the
theorem.

\item[\textbf{Case (B.ii)}] $a_{12}\neq 0.$ Taking partial derivative of
(3.11) with respect to $u$ gives%
\begin{equation}
3H_{0}[1+\left( a_{12}f^{\prime }+a_{22}g^{\prime }\right) ^{2}]^{\frac{1}{2}%
}\left[ a_{12}f^{\prime }+a_{22}g^{\prime }\right] [a_{12}f^{\prime \prime
}]=a_{12}^{2}f^{\prime \prime \prime }.  \tag{3.14}
\end{equation}%
We have again two cases:

\begin{itemize}
\item[\textbf{Case (B.ii.1)}] $f^{\prime \prime }=0.$ Then from (3.11), we
have%
\begin{equation}
\frac{H_{0}}{a_{22}}=\frac{a_{22}g^{\prime \prime }}{[1+\left(
a_{12}c_{1}+a_{22}g^{\prime }\right) ^{2}]^{\frac{3}{2}}},  \tag{3.15}
\end{equation}%
where $f^{\prime }=c_{1}.$ By solving (3.15), up to suitable translations
and constants, we obtain 
\begin{equation*}
g\left( v\right) =-\frac{1}{H_{0}}\sqrt{1-\left( \frac{H_{0}}{a_{22}}%
v\right) ^{2}}-\frac{c_{1}a_{12}}{a_{22}}v,
\end{equation*}%
which gives the statement (2.2) of the theorem.

\item[\textbf{Case (B.ii.2)}] $f^{\prime \prime }\neq 0.$ Then (3.14) can be
rewritten as 
\begin{equation}
\lbrack 1+\left( a_{12}f^{\prime }+a_{22}g^{\prime }\right) ^{2}]^{\frac{1}{2%
}}\left[ a_{12}f^{\prime }+a_{22}g^{\prime }\right] =\frac{a_{12}f^{\prime
\prime \prime }}{3H_{0}f^{\prime \prime }}.  \tag{3.16}
\end{equation}%
The partial derivative of (3.16) with respect to $v$ gives%
\begin{equation*}
1+2\left( a_{12}f^{\prime }+a_{22}g^{\prime }\right) ^{2}=0,
\end{equation*}%
which is a contradiction. This completes the proof.
\end{itemize}
\end{itemize}
\end{itemize}
\end{proof}

\section{Constant Curvature Surfaces of Type 3}

Let one translating curve be the space curve given by $\alpha=\alpha
(u)=(u,f_{1}(u),f_{2}(u))$ and another one the unit speed isotropic curve by
\begin{equation*}
\left\{ 
\begin{array}{l}
\beta =\beta (v)=(0,g_{1}(v),g_{2}(v)), \\ 
\left( g_{1}^{\prime }\right) ^{2}+\left( g_{2}^{\prime }\right) ^{2}=1,%
\text{ }g_{i}^{\prime }=\frac{dg_{i}}{dv},\text{ }i=1,2,%
\end{array}%
\right.
\end{equation*}%
where we may assume $g_{1}^{\prime }\neq 0$ without loss of generality. The
last equality yields 
\begin{equation}
g_{1}^{\prime }g_{1}^{\prime \prime }+g_{2}^{\prime }g_{2}^{\prime \prime
}=0.  \tag{4.1}
\end{equation}%
Further, since the torsion of $\alpha $ is different from zero, we get%
\begin{equation}
f_{1}^{\prime \prime }f_{2}^{\prime \prime \prime }-f_{1}^{\prime \prime
\prime }f_{2}^{\prime \prime }\neq 0,  \tag{4.2}
\end{equation}%
where $\frac{df_{i}}{du}=f_{i}^{\prime },$ etc. $i=1,2$. Thereby the
obtained translation surface belongs to type 3 and is given by%
\begin{equation}
r(u,v)=(u,f_{1}(u)+g_{1}(v),f_{2}(u)+g_{2}(v)).  \tag{4.3}
\end{equation}%
By a calculation, the Gaussian curvature is 
\begin{equation}
K=-\frac{g_{2}^{\prime \prime }}{g_{1}^{\prime }}(f_{1}^{\prime \prime
}g_{2}^{\prime }-f_{2}^{\prime \prime }g_{1}^{\prime }).  \tag{4.4}
\end{equation}

\begin{theorem} If the surface given by (4.1) has constant
Gaussian curvature $K_{0}$ in $\mathbb{G}_{3}$, then it is a generalized
cylinder with isotropic rulings ($K_{0}=0$).
\end{theorem}

\begin{proof} If $K_{0}$ vanishes then either $g_{2}^{\prime \prime }=0$
or $f_{1}^{\prime \prime }g_{2}^{\prime }-f_{2}^{\prime \prime
}g_{1}^{\prime }=0$ in (4.4). The second possibility is eliminated due to (4.2) and
thus $\beta $ becomes an isotropic line. Otherwise, $K_{0}\neq 0$, we have $%
g_{2}^{\prime \prime }\neq 0.$ Then by taking partial derivative of (4.4)
with respect to $u$, we get%
\begin{equation}
0=f_{1}^{\prime \prime \prime }g_{1}^{\prime }-f_{2}^{\prime \prime \prime
}g_{2}^{\prime }.  \tag{4.5}
\end{equation}%
From (4.2) at least one of $f_{1}^{\prime \prime \prime }$ and $%
f_{2}^{\prime \prime \prime }$ is different from zero. Thus (4.5) implies $%
g_{2}^{\prime }=cg_{1}^{\prime },$ $c\in \mathbb{R}-\{0\}.$ Considering it
into (4.1) yields a contradiction, which proves the theorem.
\end{proof} 

\begin{theorem}  If the surface given by (4.1) has constant mean
curvature $H_{0}$ in $\mathbb{G}_{3}$\ then either 

\begin{enumerate}
\item[(1)] it is either a generalized cylinder with isotropic
rulings ($H_{0}=0$); or 

\item[(2)] the translating isotropic curve is a Euclidean circle
with radius $\frac{1}{\left\vert H_{0}\right\vert }$\ ($H_{0}\neq 0$).
\end{enumerate}
\end{theorem}

\begin{proof} Assume that the surface given by (4.1) has constant mean
curvature $H_{0}$. Then we have the relation 
\begin{equation}
H_{0}=\frac{g_{2}^{\prime \prime }}{g_{1}^{\prime }},  \tag{4.6}
\end{equation}%
which immediately implies that $H_{0}$ vanishes provided $\beta $ is an
isotropic line. If $H_{0}\neq 0,$ then we have%
\begin{equation}
g_{2}^{\prime \prime }=H_{0}g_{1}^{\prime }.  \tag{4.7}
\end{equation}%
Considering (4.7) into (4.1) gives 
\begin{equation}
g_{1}^{\prime \prime }=-H_{0}g_{2}^{\prime }.  \tag{4.8}
\end{equation}%
We may formulate the equations (4.7) and (4.8) as follows: 
\begin{equation}
\left\{ 
\begin{array}{c}
g_{1}^{\prime \prime \prime }+H_{0}^{2}g_{1}^{\prime }=0, \\ 
g_{2}^{\prime \prime \prime }+H_{0}^{2}g_{2}^{\prime }=0.%
\end{array}%
\right.  \tag{4.9}
\end{equation}%
After solving (4.9) we obtain, up to suitable constants,%
\begin{equation*}
\left\{ 
\begin{array}{c}
g_{1}=\frac{c_{1}}{\left\vert H_{0}\right\vert }\sin (\left\vert
H_{0}\right\vert u)+\frac{c_{2}}{\left\vert H_{0}\right\vert }\cos
(\left\vert H_{0}\right\vert u), \\ 
g_{2}=\frac{c_{3}}{\left\vert H_{0}\right\vert }\sin (\left\vert
H_{0}\right\vert v)+\frac{c_{4}}{\left\vert H_{0}\right\vert }\cos
(\left\vert H_{0}\right\vert v).%
\end{array}%
\right.
\end{equation*}%
Since $(g_{1}^{\prime })^{2}+(g_{2}^{\prime })^{2}=1,$ we have $%
(c_{1})^{2}+(c_{3})^{2}=1,$ $(c_{2})^{2}+(c_{4})^{2}=1$ and $%
c_{1}c_{2}+c_{3}c_{4}=0.$ This means that $\beta $ is a Euclidean circle
with radius $\frac{1}{\left\vert H_{0}\right\vert }.$
\end{proof} 

\section{Constant Curvature Surfaces of Type 4}

In last section, we are interested in the surfaces generated by
translating a space curve $\alpha=\alpha (u)=(u,f_{1}(u),f_{2}(u))$ and a planar non-isotropic curve $\beta =\beta (v)=(v,g(v),av),\text{ }a\in \mathbb{R}.$
Since the torsion of $\alpha $ is different from zero, we have%
\begin{equation}
f_{1}^{\prime \prime }f_{2}^{\prime \prime \prime }-f_{1}^{\prime \prime
\prime }f_{2}^{\prime \prime }\neq 0,  \tag{5.1}
\end{equation}%
where $\frac{df_{i}}{du}=f_{i}^{\prime }$ and so on, $i=1,2$. Therefore the obtained
translation surface is of the form%
\begin{equation}
r(u,v)=(u+v,f_{1}(u)+g(v),f_{2}(u)+av).  \tag{5.2}
\end{equation}%
By a calculation, the Gaussian curvature turns to%
\begin{equation}
K=\frac{g^{\prime \prime }\left[ f_{1}^{\prime \prime }\left( f_{2}^{\prime
}-a\right) ^{2}-f_{2}^{\prime \prime }(f_{2}^{\prime }-a)(f_{1}^{\prime
}-g^{\prime })\right] }{\left[ (f_{2}^{\prime }-a)^{2}+\left( f_{1}^{\prime
}-g^{\prime }\right) ^{2}\right] ^{2}}.  \tag{5.3}
\end{equation}

\begin{theorem}  If the surface given by (5.2) has constant Gaussian
curvature $K_{0}$\ in $\mathbb{G}_{3},$\ then it is a generalized cylinder
with non-isotropic rulings ($K_{0}=0$).
\end{theorem}

\begin{proof} We divide the proof into two cases:

\begin{itemize}
\item[\textbf{Case (A)}] $K_{0}=0$. From (5.3), we conclude either $g^{\prime
\prime }=0$, namely the surface is generalized cylinder with non-isotropic
rulings, or 
\begin{equation}
f_{1}^{\prime \prime }(f_{2}^{\prime }-a)-f_{2}^{\prime \prime
}(f_{1}^{\prime }-g^{\prime })=0. \tag{5.4}
\end{equation}%
Taking partial derivative of (5.4) with respect to $v$, we get $%
f_{2}^{\prime \prime }=0$, which is not possible due to (5.1).

\item[\textbf{Case (B)}] $K_{0}\neq 0$. By taking twice partial derivative
of (5.3) with respect to $v$, we deduce%
\begin{equation}
\left. 
\begin{array}{c}
-4K_{0}\left[ 3\left( f_{1}^{\prime }-g^{\prime }\right) ^{2}+\left(
f_{2}^{\prime }-a\right) ^{2}\right] = \\ \frac{1}{g^{\prime \prime }}\left( 
\frac{g^{\prime \prime \prime }}{g^{\prime \prime }}\right) ^{\prime }\left[
f_{1}^{\prime \prime }\left( f_{2}^{\prime }-a\right) ^{2}-f_{2}^{\prime
\prime }(f_{2}^{\prime }-a)(f_{1}^{\prime }-g^{\prime })\right] 
+2\frac{g^{\prime \prime \prime }}{g^{\prime \prime }}(f_{2}^{\prime
}-a)f_{2}^{\prime \prime }%
\end{array}%
\right.  \tag{5.5}
\end{equation}%
where $g^{\prime \prime }\neq 0$ due to our assumption. Put $\zeta =\frac{1}{%
g^{\prime \prime }}(\frac{g^{\prime \prime \prime }}{g^{\prime \prime }}%
)^{\prime }$ into (5.5). After taking partial derivative of (5.5) with respect to $%
v$, we conclude%
\begin{equation}
24K_{0}(f_{1}^{\prime }-g^{\prime })=\frac{\zeta ^{\prime }}{g^{\prime
\prime }}\left[ f_{1}^{\prime \prime }\left( f_{2}^{\prime }-a\right)
^{2}-f_{2}^{\prime \prime }(f_{2}^{\prime }-a)(f_{1}^{\prime }-g^{\prime })%
\right] +3\zeta (f_{2}^{\prime }-a)f_{2}^{\prime \prime },  \tag{5.6}
\end{equation}%
where $\zeta ^{\prime }=\frac{d\zeta }{dv}.$ The partial derivative of (5.6)
with respect to $v$ implies%
\begin{equation}
-\frac{24}{f_{2}^{\prime \prime }(f_{2}^{\prime }-a)}=\frac{1}{g^{\prime
\prime }}\left( \frac{\zeta ^{\prime }}{g^{\prime \prime }}\right) ^{\prime }%
\left[ \frac{f_{1}^{\prime \prime }(f_{2}^{\prime }-a)}{f_{2}^{\prime \prime
}}-(f_{1}^{\prime }-g^{\prime })\right] +3\frac{\zeta ^{\prime }}{g^{\prime
\prime }}.  \tag{5.7}
\end{equation}%
After again taking partial derivative of (5.7) with respect to $u$ and $v$,
we deduce%
\begin{equation}
0=\left( \frac{1}{g^{\prime \prime }}\left( \frac{\zeta ^{\prime }}{%
g^{\prime \prime }}\right) ^{\prime }\right) ^{\prime }\left[ \left( \frac{%
f_{1}^{\prime \prime }(f_{2}^{\prime }-a)}{f_{2}^{\prime \prime }}\right)
^{\prime }-f_{1}^{\prime \prime }\right] .  \tag{5.8}
\end{equation}%
We have two cases to solve (5.8):

\begin{itemize}
\item[\textbf{Case (B.i)}] $\left( \frac{\zeta ^{\prime }}{g^{\prime \prime }%
}\right) ^{\prime }=c_{3}g^{\prime \prime }$. Up to suitable constant, we
have $\frac{\zeta ^{\prime }}{g^{\prime \prime }}=c_{3}g^{\prime }$.
Substituting these into (5.7) gives%
\begin{equation*}
-\frac{24}{f_{2}^{\prime \prime }(f_{2}^{\prime }-a)}=c_{3}\frac{%
f_{1}^{\prime \prime }(f_{2}^{\prime }-a)}{f_{2}^{\prime \prime }}%
-c_{3}f_{1}^{\prime }+4c_{3}g^{\prime },
\end{equation*}%
which implies $c_{3}=0$ and thus $\zeta ^{\prime }=0.$ Considering it into (5.6)
leads to%
\begin{equation}
24K_{0}(f_{1}^{\prime }-g^{\prime })=3c_{4}(f_{2}^{\prime }-a)f_{2}^{\prime
\prime },  \tag{5.9}
\end{equation}%
where $\zeta =c_{4}.$ (5.9) yields a contradiction due to $K_{0}\neq 0.$

\item[\textbf{Case (B.ii)}] $\left( \frac{f_{1}^{\prime \prime
}(f_{2}^{\prime }-a)}{f_{2}^{\prime \prime }}\right) ^{\prime
}-f_{1}^{\prime \prime }=0$. Up to suitable constant, we have%
\begin{equation}
\frac{f_{1}^{\prime \prime }}{f_{1}^{\prime }}=\frac{f_{2}^{\prime \prime }}{%
f_{2}^{\prime }-a}.  \tag{5.10}
\end{equation}%
After solving (5.10) we obtain $f_{1}^{\prime \prime }=c_{5}f_{2}^{\prime
\prime }$ which is a contradiction due to (5.1). Therefore the proof is
completed.
\end{itemize}
\end{itemize}
\end{proof}

By a calculation, the mean curvature turns to%
\begin{equation}
H=\frac{\left( f_{2}^{\prime }-a\right) g^{\prime \prime }+\left(
f_{2}^{\prime }-a\right) f_{1}^{\prime \prime }-\left( f_{1}^{\prime
}-g^{\prime }\right) f_{2}^{\prime \prime }}{\left[ \left( f_{2}^{\prime
}-a\right) ^{2}+\left( f_{1}^{\prime }-g^{\prime }\right) ^{2}\right] ^{%
\frac{3}{2}}}.  \tag{5.11}
\end{equation}%
First we investigate the minimality case:

\begin{theorem} There does not exist a minimal translation surface
given by (5.2) in $\mathbb{G}_{3}$.
\end{theorem}

\begin{proof} Let us assume the contrary situation. Then (5.11) reduces to%
\begin{equation}
\left( f_{2}^{\prime }-a\right) \left( g^{\prime \prime }+f_{1}^{\prime
\prime }\right) -\left( f_{1}^{\prime }-g^{\prime }\right) f_{2}^{\prime
\prime }=0.  \tag{5.12}
\end{equation}%
The partial derivative of (5.12) with respect to $v$ yields%
\begin{equation}
\left( f_{2}^{\prime }-a\right) g^{\prime \prime \prime }+f_{2}^{\prime
\prime }g^{\prime \prime }=0.  \tag{5.13}
\end{equation}%
We have two cases:

\begin{itemize}
\item[\textbf{Case (A)}] $g^{\prime }=c_{1},$ $c_{1}\in \mathbb{R}.$ Then
(5.13) turns to%
\begin{equation}
\frac{f_{1}^{\prime \prime }}{f_{1}^{\prime }-c_{1}}=\frac{f_{2}^{\prime
\prime }}{f_{2}^{\prime }-a}  \tag{5.14}
\end{equation}%
and solving (5.14) yields $f_{1}^{\prime \prime }=c_{2}f_{2}^{\prime \prime
},$ $c_{2}\in \mathbb{R-}\left\{ 0\right\} .$ This leads to a contradiction
due to (5.1).

\item[\textbf{Case (B)}] $g^{\prime \prime }\neq 0.$ Then (5.13) can be
rewritten as%
\begin{equation}
\frac{g^{\prime \prime \prime }}{g^{\prime \prime }}=c_{3}=\frac{%
-f_{2}^{\prime \prime }}{f_{2}^{\prime }-a},\text{ }c_{3}\in \mathbb{R-}%
\left\{ 0\right\} .  \tag{5.15}
\end{equation}%
which implies $g^{\prime \prime }=c_{3}g^{\prime },$ up to suitable
constant. Substituting these into (5.12) gives%
\begin{equation}
f_{1}^{\prime \prime }+c_{3}f_{1}^{\prime }=0.  \tag{5.16}
\end{equation}%
From (5.15) and (5.16) we derive%
\begin{equation*}
f_{2}^{\prime \prime \prime }=-c_{3}f_{2}^{\prime \prime }\text{ \ and \ }%
f_{1}^{\prime \prime \prime }=-c_{3}f_{1}^{\prime \prime },
\end{equation*}%
which is no possible due to (5.1). Therefore the proof is completed.
\end{itemize}
\end{proof}

\begin{theorem} If the surface given by (5.2) has nonzero constant
mean curvature $H_{0}$\ in $\mathbb{G}_{3}$, then it is a generalized
cylinder with non-isotropic rulings whose the base curve satisfies the
equation%
\begin{equation*}
f_{1}=cu+H_{0}^{2}\left\{ \frac{1}{2}\left( f_{2}-au\right) ^{2}\zeta
(\sigma )-\frac{1}{2}\int \left[ \left( f_{2}-au\right) ^{2}\frac{d\zeta
(\sigma )}{du}\right] du\right\} ,
\end{equation*}%
where $c\in R$\ and%
\begin{equation*}
\zeta (\sigma )=\int \left( f_{2}-au\right) d\sigma \text{ \ }for\text{ \ }%
\sigma =\frac{f_{1}^{\prime }-c_{1}}{f_{2}^{\prime }-a}.
\end{equation*}%
\end{theorem} 

\begin{proof} The partial derivative of (5.11) with respect to $v$ gives%
\begin{equation}
3H_{0}\left[ \left( f_{2}^{\prime }-a\right) ^{2}+\left( f_{1}^{\prime
}-g^{\prime }\right) ^{2}\right] ^{\frac{1}{2}}\left( f_{1}^{\prime
}-g^{\prime }\right) g^{\prime \prime }=\left( f_{2}^{\prime }-a\right)
g^{\prime \prime \prime }+f_{2}^{\prime \prime }g^{\prime \prime }. 
\tag{5.17}
\end{equation}%
To solve (5.17), we have two cases:

\begin{itemize}
\item[\textbf{Case (A)}] $g^{\prime }=c_{1},$ $c_{1}\in \mathbb{R}$. (5.11)
turns to%
\begin{equation}
H_{0}=\frac{\left( f_{2}^{\prime }-a\right) f_{1}^{\prime \prime }-\left(
f_{1}^{\prime }-c_{1}\right) f_{2}^{\prime \prime }}{\left[ \left(
f_{2}^{\prime }-a\right) ^{2}+\left( f_{1}^{\prime }-c_{1}\right) ^{2}\right]
^{\frac{3}{2}}}.  \tag{5.18}
\end{equation}%
Put $\sigma =\dfrac{f_{1}^{\prime }-c_{1}}{f_{2}^{\prime }-a}$ into (5.18).
Then we get 
\begin{equation}
H_{0}\left( f_{2}^{\prime }-a\right) =\frac{\frac{d\sigma }{du}}{\left(
1+\sigma ^{2}\right) ^{\frac{3}{2}}}.  \tag{5.19}
\end{equation}%
Up to suitable constant, an integration of (5.19) with respect to $u$ gives%
\begin{equation}
H_{0}\left( f_{2}-au\right) =\frac{\sigma }{\left( 1+\sigma ^{2}\right) ^{%
\frac{1}{2}}}.  \tag{5.20}
\end{equation}%
Again an integration of (5.20) with respect to $\sigma $, we conclude%
\begin{equation}
H_{0}\int \left( f_{2}-au\right) d\sigma =\sqrt{1+\sigma ^{2}}.  \tag{5.21}
\end{equation}%
Substituting (5.21) into (5.20) yields%
\begin{equation*}
H_{0}^{2}\left( f_{2}-au\right) \int \left( f_{2}-au\right) d\sigma =\sigma ,
\end{equation*}%
or%
\begin{equation}
f_{1}^{\prime }-c_{1}=H_{0}^{2}\left( f_{2}-au\right) \left( f_{2}^{\prime
}-a\right) \zeta (\sigma ),  \tag{5.22}
\end{equation}%
where $\zeta (\sigma )=\int \left( f_{2}-au\right) d\sigma $. The partial
integration of (5.22) with respect to $u$ gives%
\begin{equation*}
f_{1}=c_{1}u+H_{0}^{2}\left\{ \frac{1}{2}\left( f_{2}-au\right) ^{2}\zeta
(\sigma )-\frac{1}{2}\int \left\{ \left( f_{2}-au\right) ^{2}\frac{d\zeta
(\sigma )}{du}\right\} du\right\} .
\end{equation*}

\item[\textbf{Case (B)}] $g^{\prime \prime }\neq 0$. (5.17) can be rewritten
as%
\begin{equation}
3H_{0}\left[ \left( f_{2}^{\prime }-a\right) ^{2}+\left( f_{1}^{\prime
}-g^{\prime }\right) ^{2}\right] ^{\frac{1}{2}}\left( f_{1}^{\prime
}-g^{\prime }\right) =\left( f_{2}^{\prime }-a\right) \frac{g^{\prime \prime
\prime }}{g^{\prime \prime }}+f_{2}^{\prime \prime }.  \tag{5.23}
\end{equation}%
The partial derivative of (5.23) with respect to $v$ gives%
\begin{equation}
\frac{2\frac{(f_{1}^{\prime }-g^{\prime })^{2}}{f_{2}^{\prime }-a}%
+f_{2}^{\prime }-a}{\left[ \left( f_{2}^{\prime }-a\right) ^{2}+\left(
f_{1}^{\prime }-g^{\prime }\right) ^{2}\right] ^{\frac{1}{2}}}=-\frac{1}{%
3H_{0}g^{\prime \prime }}\left( \frac{g^{\prime \prime \prime }}{g^{\prime
\prime }}\right) ^{\prime }.  \tag{5.26}
\end{equation}%
By again taking partial derivative of (5.26) with respect to $u$ we derive a
polynomial equation on $\left( f_{1}^{\prime }-g^{\prime }\right) .$ In that
equation, the coefficient of the term of highest degree is $\left(
f_{2}^{\prime }-a\right) f_{2}^{\prime \prime }$. This one cannot vanish due
to (5.1) and therefore we achieve a contradiction which completes the proof.
\end{itemize}
\end{proof}

\section{Conclusions}

This study is devoted to obtain the translation surfaces in $\mathbb{%
G}_{3}$ with $K=const.$ and $H=const.$ when at least one of the translating curves is planar. In this sense, to classify
the surfaces  in $\mathbb{G}_{3}$ whose both of translating curves are
non-planar is still an open problem, that is not easy to solve. However, it is obvious that such a surface  can be neither flat nor minimal (see Theorem 1.1). Consequently, the known results can be summarized as in Table 1:

\bigskip

\end{document}